\documentclass[12pt, reqno]{amsart}
\usepackage{amsmath, amsthm, amscd, amsfonts, amssymb, graphicx, color}
\usepackage[bookmarksnumbered, colorlinks, plainpages]{hyperref}

\makeatletter \oddsidemargin.9375in \evensidemargin \oddsidemargin
\newtheorem{theorem}{Theorem}[section]
\newtheorem{lemma}[theorem]{Lemma}

\theoremstyle{definition}

\theoremstyle{remark}

\numberwithin{equation}{section}

\begin{document}
\setcounter{page}{1}

%%%%%%%%%%%%%%%%%%%%%%%%%%%%%%%%%%%%%%%%%%%%%%%
%% Please do not remove the following statement.
%%%%%%%%%%%%%%%%%%%%%%%%%%%%%%%%%%%%%%%%%%%%%%%
%\noindent \textbf{{\footnotesize \\ Vol. XX, No XX, (201X), pp XX-XX}}\\[1.00in]
%%%%%%%%%%%%%%%%%%%%%%%%%%%%%%%%%%%%%%%%%%%%%%%

%%%%%%%%%%%%%%%%%%%%%%%%%%%%%%%%%%%%%%%%%%%%%%%%%%%%%%%%%%%%%%%%%%%%%
% Insert title of your article. Note: \title[short title]{full title}
%%%%%%%%%%%%%%%%%%%%%%%%%%%%%%%%%%%%%%%%%%%%%%%%%%%%%%%%%%%%%%%%%%%%%
\title[Commutative weakly nil-neat group rings]{Commutative weakly nil-neat \\ group rings}
%%%%%%%%%%%%%%%%%%%%%%%%%%%%%%%%%%%%%%
% Author's name must be inserted here
%%%%%%%%%%%%%%%%%%%%%%%%%%%%%%%%%%%%%%
\author[M. Samiei]{Mahdi Samiei}
\author[P. Danchev]{Peter Danchev$^{*}$}

%%%%%%%%%%%%%%%%%%%%%%%%
\thanks{{\scriptsize
\hskip -0.4 true cm MSC(2010): 20C07; 16U60; 16S34.
$*$Corresponding author}}

%%%%%%%%%%%%%%%%%%%%%%%%%%%%%%%%%%%%%%%%%%%
\begin{abstract}
Let $R$ be a ring and let $G$ be a group. We prove a rather curious necessary and sufficient condition for the commutative group ring $RG$ to be {\it weakly nil-neat} only in terms of $R$,$G$ and their sections. This somewhat expands three recent results, namely those established by McGovern et al. in (J. Algebra Appl., 2015), by Danchev-McGovern in (J. Algebra, 2015) and by the present authors in (J. Math. Univ. Tokushima, 2019), related to commutative nil-clean, weakly nil-clean and nil-neat group rings, respectively.
\end{abstract}

%%%%%%%%%%%%%%%%%%%%%%%
\maketitle
%%%%%%%%%%%%%%%%%%%%%%%

\section{Introduction and the main result}

All given rings are assumed to be commutative rings containing non-zero identity as well as all given groups are assumed to be abelian written multiplicatively. For such a ring $R$, the notation $J(R)$ denotes the Jacobson radical of $R$ and $N(R)$ denotes the nil-radical of $R$ as it is well known that the inclusion $N(R)\subseteq J(R)$ holds, while for such a group $G$, the notation $G_p$ denotes the $p$-torsion component of $G$ for some fixed prime integer $p$. We shall say that $G$ is a $p$-group whenever $G=G_p$. Suppose $RG$ is now the group ring of $G$ over $R$ defined as usual.

A ring $R$ is called {\it nil-clean} if each its element is a sum of a nilpotent and an idempotent. These rings were intensively investigated by too many authors (see, e.g., \cite{DM} and the bibliography therewith). A valuable criterion for a commutative group ring to be nil-clean was successfully obtained in \cite{M}. Indeed, it was shown there that {\it a commutative group ring $RG$ is nil-clean precisely when $G$ is a $2$-group and $R$ is a nil-clean ring}. Moreover, as a common generalization to this class of rings, in \cite{DM} were supposed for an investigation the so-called {\it weakly nil-clean rings} that are rings whose elements are sums or differences of nilpotents and idempotents. Clearly, all nil-clean rings are weakly nil-clean, while the converse is wrong as the example of the field of three elements $\mathbb{Z}_3$ shows. However, the direct product $\mathbb{Z}_3\times \mathbb{Z}_3$ is manifestly {\it not} weakly nil-clean. A criterion for a commutative group ring to be weakly nil-clean as subsequently extracted there (see Lemma~\ref{1.17} stated below).

On the other vein, the class of {\it nil-neat rings} was defined in \cite{S} as those rings which are the proper homomorphic images of all nil-clean rings. Thereby, nil-clean rings are themselves nil-neat, whereas the converse containment is not valid as \cite[Example 2.10]{S} illustrates. About the "group ring problem", it was established in \cite{SD2} that {\it a commutative group ring $RG$ is nil-neat exactly when $G=\{1\}$ and $R$ is nil-neat, or $G\not=\{1\}$ is a $2$-group and $R$ is a nil-clean ring}. Thus, it is a proper nil-neat ring (that is, it is nil-neat but {\it not} nil-clean) uniquely when $G$ is the trivial group.

On the other hand, by analogy, the class of {\it weakly nil-neat rings} was naturally introduced in \cite{SD1} as those rings which are the proper homomorphic images of all weakly nil-clean rings. Evidently, all weakly nil-clean and nil-neat rings are weakly nil-neat, but the reverse implication is untrue as the example $\mathbb{Z}_3\times \mathbb{Z}_3$ demonstrates.

So, our motivation in writing up this brief research note is to extend the listed above achievements concerning the cited above "group ring problem" to this new point of view. In fact, we will prove that the group ring $RG$ is weakly nil-neat in the way indicated in \cite{SD2} by showing that $RG$ is either proper weakly nil-neat (in the sense that it is weakly nil-neat but {\it not} weakly nil-clean) only when $G=\{1\}$, or it is weakly nil-clean otherwise. Some other closely related to that type results can be found in \cite{U}, although there is no a final result of this type received there.

\medskip

The next three statements are useful for our further considerations and applicable purposes.

\begin{lemma}(\cite{DM}, Theorem $1.17$)\label{1.17}
Let $R$ be a ring. The following four statements are equivalent:
\begin{itemize}
\item[(1)]
$R$ is a weakly nil-clean ring.
\item[(2)]
$R$ is zero-dimensional and there is at most one maximal ideal of $R$, say $M$, which satisfies $R/M\cong\mathbb{Z}_3$ while for all other maximal ideals $N$ of $R$ we have $R/N\cong\mathbb{Z}_2$.
\item[(3)]
$R/N(R)$ is isomorphic to either a boolean ring, or to $\mathbb{Z}_3$, or to the product of two such rings.
\item[(4)]
$J(R)$ is nil and $R/J(R)$ is isomorphic to either a boolean ring, or to $\mathbb{Z}_3$, or to the product of two such rings. In particular, $J(R)=N(R)$.
\end{itemize}
\end{lemma}

The following assertion is a simple restatement in an equivalent form of the original result, which setting will be more convenient for us.

\begin{lemma}(\cite{DM}, Theorem $2.1$)\label{2.1}
Let $R$ be a ring and $G$ a group. The group ring $RG$ is weakly nil-clean if, and only if, exactly one of the following three conditions is satisfied:
\begin{itemize}
\item[(1)]
$R$ is nil-clean and $G$ is a non-trivial torsion $2$-group;
\item[(2)]
$R$ is weakly nil-clean with $3\in N(R)$ (or, equivalently, $R/N(R)\cong \mathbb{Z}_3$) and $G$ is a non-trivial torsion $3$-group;
\item[(3)]
$R$ is weakly nil-clean and $G$ is trivial.
\end{itemize}
\end{lemma}

The following affirmation is pivotal.

\begin{lemma}(\cite{SD1}, Theorem $2.7$)\label{2.7}
A ring $R$ is weakly nil-neat if, and only if, either $R$ is a field, or $R/J(R)$ is isomorphic to either a boolean ring, or $\mathbb{Z}_3$, or the direct product of two such rings when $J(R)\neq 0$, or $R$ is isomorphic to a subring of a direct product of copies of $\mathbb{Z}_2$ and/or at most two copies of $\mathbb{Z}_3$ when $J(R)=0$ and, moreover, every nonzero prime ideal of $R$ is maximal.
\end{lemma}

There exists a (weakly) nil-neat ring which is {\it not} (weakly) nil-clean. In fact, as already noted above, the double direct product $\mathbb{Z}_3\times \mathbb{Z}_3$ will be workable for this aim.

\medskip

A crucial formula in our argumentation will be the following one due to Karpilovsky (see \cite{K}): For a ring $R$ and a group $G$, the next equality is true:

$$
J(RG)=J(R)G+\langle r(g_p-1) ~ | ~ r\in R, pr\in J(R), g_p\in G_p\rangle.
$$

\medskip

So, we are now ready to establish the following criterion for the group ring to be weakly nil-neat, which sounds like this.

\begin{theorem}
Let $R$ be a ring and $G$ a group. The group ring $RG$ is weakly nil-neat if, and only if, exactly one of the following four conditions is satisfied:
\begin{itemize}
\item[(1)]
$G$ is trivial and $R$ is weakly nil-neat.
%\item[(2)]
%$G$ is non-trivial such that $G$ is a torsion $2$-group and $R$ is weakly nil-clean.
\item[(2)]
$G$ is non-trivial such that $G$ is a torsion $2$-group and $R$ is nil-clean (in fact, $R/N(R)$ is boolean).
\item[(3)]
$G$ is non-trivial such that $G$ is a torsion $3$-group and $R$ is weakly nil-clean with $3\in N(R)$ (in fact, $R/N(R)$ is isomorphic to $\mathbb{Z}_3$).
\item[(4)]
$G$ is cyclic of order $2$ and $R\cong \mathbb{Z}_3$.
\end{itemize}
\end{theorem}

\begin{proof} "$\Rightarrow$". Suppose that $RG$ is a weakly nil-neat ring. It is easy to verify that $R$ is weakly nil-neat whenever $G$ is trivial, because then $RG\cong R$. So, we shall assume hereafter that $G$ is non-trivial. Hence, considering the standard augmentation map $RG\to R$, one finds that $R$ is weakly nil-clean as a proper homomorphic image of $RG$.

Firstly, if $R$ is not reduced, then by Lemma~\ref{1.17}$(4)$ we have that $J(R)\neq \{0\}$ and that $R/J(R)$ is isomorphic to either $B$, or $\mathbb{Z}_3$, or $B\times \mathbb{Z}_3$, where $B$ is a boolean ring. Thus $R$ has either $B$, or $\mathbb{Z}_3$, or $B\times \mathbb{Z}_3$ as a proper epimorphic image, whence either $BG$, or $\mathbb{Z}_3G$, or $(B\times \mathbb{Z}_3)G$ is such a proper homomorphic image for $RG$. Since $B\times \mathbb{Z}_3$ is obviously weakly nil-clean but not nil-clean and $3\not\in N(B\times \mathbb{Z}_3)=\{0\}$ whenever $B$ is a non-trivial direct factor, that is, $3\not=0$ in this direct product, by combining Lemma~\ref{2.1} and the fact that every proper homomorphic image of a weakly nil-neat ring is weakly nil-clean, we conclude that either $BG$ or $\mathbb{Z}_3G$ is weakly nil-clean. Again by making use of Lemma~\ref{2.1}, we get $G$ is either a non-trivial torsion $2$-group or a non-trivial torsion $3$-group, respectively. About $R$, the only two possibilities were that $R$ is nil-clean (in fact, $R/N(R)\cong B$) or that $R$ is weakly nil-clean with $3\in Nil(R)$ (in fact, $R/N(R)\cong \mathbb{Z}_3$), respectively.

Secondly, if now $R$ is reduced, Lemma~\ref{1.17}$(4)$ shows that $R$ is isomorphic to either $B$, or to $\mathbb{Z}_3$, or to $B\times \mathbb{Z}_3$, where $B$ is a boolean ring. We will prove in what follows that these three mutually exclusive cases lead to the truthfulness of the stated above four conditions in the theorem.

\medskip

\textbf{Case $1$:} $R\cong B$ is boolean.
\begin{itemize}
\item[(1.1)]
If $G_2$ is non-trivial, then $B(G/G_2)$ is a weakly nil-clean ring being a proper homomorphic image of $BG$. In this case, we must have $G/G_2\cong \{1\}$ utilizing Lemma~\ref{2.1}$(3)$ and thus $G=G_2$. Let us notice that Lemma~\ref{2.1}$(1)$ would imply that $G/G_2$ is a non-trivial $2$-group, which is pretty impossible as $G/G_2$ being a $2$-group yields at once that $G$ has to be a $2$-group, that is, $G=G_2$, a contradiction.
\item[(1.2)]
If now $G_2$ is trivial, we can get that $G$ is either torsion-free, that is, $G_t:=\prod_{\forall p} G_p=\{1\}$ or that $G$ is $q$-torsion for some single prime $q$ different from $2$. In fact, if there are two distinct primes $p,q\not= 2$ such that $G_p,G_q$ are non-identity, then both $B(G/G_p)$ and $B(G/G_q)$ will be weakly nil-clean, whence using items (1) or (3) of Lemma~\ref{2.1} we will infer that $G/G_p$ and $G/G_q$ are either non-trivial $2$-groups or that they are simultaneously trivial. In the first possibility, we detect that $G=G_p=G_q$ as both $G_p$ and $G_q$ are $2$-divisible (see \cite{F}) which is wrong contradicting with the non-triviality of the both quotients. In the other case, we again will have that $G=G_p=G_q$ which, in turn, will yield that $G=\{1\}$ as $G_p\cap G_q=\{1\}$, which is in sharp contrary to the assumption that $G\not=\{1\}$ stated at the beginning of our proof.
That is why, one extracts that either $G$ is torsion-free or is a $q$-group, thus substantiating our claim.

So, one follows now by virtue of the listed above explicit formula from \cite[Theorem]{K} that $J(BG)=0$. Indeed, $J(B)=\{0\}$, and since $(q,2)=1$, the integer $q$ will invert in $B$ as $B$ is of characteristic $2$, so that $qr=0$ for any $r\in B$ will imply that $r=0$, as needed.

Furthermore, with Lemma~\ref{2.7} at hand, $BG$ being of characteristic $2$ could be only a boolean ring (as it is isomorphic to a subring of a direct product of copies of $\mathbb{Z}_2$). This immediately assures that, for any $g\in G\subseteq BG$, the equality $g^2=g$ holds, whence $g=1$, which is demonstrably against our initial assumption on the non-triviality of $G$. Notice that $BG$ is unable to be a field, which direct check we leave to the interested reader.
\end{itemize}

\textbf{Case $2$:} $R\cong \mathbb{Z}_3$.
\begin{itemize}
\item[(2.1)]
Suppose that $G_3$ is non-trivial. Then $\mathbb{Z}_3(G/G_3)$ is a weakly nil-clean ring and so either $G/G_3$ is a non-trivial torsion $3$-group by Lemma~\ref{2.1}$(2)$, which is obviously non-sense as it will lead to $G=G_3$ and thus to the absurd that $G/G_3\cong \{1\}$, or $G/G_3$ is trivial by Lemma~\ref{2.1}$(3)$, which leads us to the desired equality $G=G_3$.
\item[(2.2)]
Let $G_3$ be now trivial. We shall show as above in case (1.2) that $G$ is either torsion-free or a $q$-group for some single prime $q$ different to $3$. Suppose, on the contrary, that $G_q$ is not equal to $\{1\}$, for any prime $q$ different to $3$. Consequently, $\mathbb{Z}_3(G/G_q)$ is a weakly nil-clean ring as being a proper homomorphic image of $\mathbb{Z}_3G$. Therefore, Lemma~\ref{2.1}$(2)$ applies to get that $G/G_q$ is a non-trivial $3$-group and so we deduce that $G=G_q$, because each $q$-torsion component is $3$-divisible, i.e., $(G_q)^{3} = G_q$ (see, for instance, \cite{F}). This equality for $G$ contradicts, certainly, to the non-triviality of the factor-group $G/G_q$. That is why, Lemma~\ref{2.1}$(3)$ is now applicable to get that $G/G_q$ has to be the identity group which, indeed, assures that $G=G_q$. But the validity of the equality $G=G_q=G_p$ for two distinct primes $q,p\not= 3$ means that $G=G_p\cap G_q=\{1\}$, so that we find that $G=\{1\}$, which is manifestly untrue. Thus, one infers that either $G_t:=\coprod_{\forall p} G_p=\{1\}$ or $G=G_q$ is $q$-torsion for some integer $q$.

Hence, in view of the already cited above formula in \cite[Theorem]{K}, we derive that $J(\mathbb{Z}_3G)=0$ by taking into account that $J(\mathbb{Z}_3)=\{0\}$ and that $qr=0$ for any $r\in \mathbb{Z}_3$ insures $r=0$ since $q$ is a unit in $\mathbb{Z}_3$.

Further, with Lemma~\ref{2.7} at hand, $\mathbb{Z}_3G$ being of characteristic $3$ could be isomorphic to only $\mathbb{Z}_3$ or to $\mathbb{Z}_3\times \mathbb{Z}_3$; actually, it is worth noticing here that the property of being isomorphic to a subring of a direct product (= subdirect product) can be interpreted as a direct isomorphism. A straightforward check is a guarantor that the elements of these two rings satisfy the equation $x^3=x$. So, in these two cases, we arrive at the fact that, for any $g$ in $G\subseteq \mathbb{Z}_3G$, the equality $g^3=g$ is valid. It equivalently forces that $g^2=1$.

Certainly, if $\mathbb{Z}_3G\cong \mathbb{Z}_3$, we may conclude that $G$ is trivial, which once again contradicts our former assumption. Even something more, $\mathbb{Z}_3G$ is unable to be any field, which direct check we leave to the interested reader.

If now the isomorphism $\mathbb{Z}_3G\cong \mathbb{Z}_3\times \mathbb{Z}_3$ is fulfilled, we can say that $G=\langle g\rangle =\{1,g~|~g^2=1\}$. In fact, $G$ is necessarily finite (as for otherwise, $\mathbb{Z}_3G$ will be infinite which is impossible. Therefore, since $|\mathbb{Z}_3G|=|\mathbb{Z}_3|^{|G|}=|\mathbb{Z}_3\times \mathbb{Z}_3|$, one extracts that $3^{|G|}=9$ yielding that $|G|=2$, as expected.
\end{itemize}

\textbf{Case $3$:} $R\cong B\times\mathbb{Z}_3$.
\begin{itemize}
\item[(3.1)]
Let $G_3\neq \{1\}$. Then the group ring $(B\times\mathbb{Z}_3)(G/G_3)$ is weakly nil-clean as being a proper homomorphic image of $(B\times\mathbb{Z}_3)G$ and so $G/G_3$ has to be trivial in virtue of Lemma~\ref{2.1}$(3)$. So, $G=G_3$, as required. Consequently, $(B\times \mathbb{Z}_3)G_3\cong BG_3\times \mathbb{Z}_3G_3$ being a weakly nil-neat ring insures that both $BG_3$ and $\mathbb{Z}_3G_3$ are weakly nil-clean rings as proper epimorphic images of the whole weakly nil-neat ring. However, according to Lemma~\ref{2.1}(1), $BG_3$ cannot be such a ring, so that this case is unavailable.

It is worthwhile noticing that the same conclusion may be drawn provided $G_2\neq \{1\}$, as in that case $\mathbb{Z}_3G_2$ need not be weakly nil-clean owing to Lemma~\ref{2.1}(2).

Another approach might be that we definitely will have that $G=G_2=G_3\not= \{1\}$, which is an absurd since $G_2\cap G_3=\{1\}$.
\item[(3.2)]
Assume now that $G_2=G_3=\{1\}$. Again adapting the idea of the previous point quoted above, we shall detect that $J((B\times \mathbb{Z}_3)G)=\{0\}$. Indeed, in order to show that, one needs to get that either $G$ is torsion-free or is a $q$-group for some single prime $q$ different to $2$ and $3$. For otherwise, assume $G_q$ is not equal to $\{1\}$. Thus, $(B\times\mathbb{Z}_3)(G/G_q)$ is weakly nil-clean as being a proper homomorphic image of $(B\times\mathbb{Z}_3)G$. It follows from Lemma \ref{2.1}$(3)$ that $G/G_q$ has to be trivial, i.e., $G=G_q$. Thus, if there exist two such different primes $p$ and $q$ that $G=G_p=G_q$, the group $G$ must be trivial, because $G_q\cap G_p=\{1\}$. This is, of course, a contradiction with our frontier's assumption that $G\not=\{1\}$. Thereby, the claim about $G$ sustained.

Furthermore, with the given above pivotal formula from \cite{K} at hand, we may repeat the same trick already demonstrated above to get that the group ring $(B\times \mathbb{Z}_3)G$ is semi-primitive (= semi-simple in the sense of Jacobson) as $J(B\times \mathbb{Z}_3)\cong J(B)\times J(\mathbb{Z}_3)=\{0\}$ and $qr=0$ for any $r\in B\times \mathbb{Z}_3$ enables us that $r=0$, because the characteristic of $B\times \mathbb{Z}_3$ is exactly $6$, as required.

After bearing that in mind, we may write in accordance with Lemma~\ref{2.7} that $(B\times \mathbb{Z}_3)G$ is isomorphic to a subring of one of direct products $B'\times \mathbb{Z}_3$ and $B'\times \mathbb{Z}_3\times \mathbb{Z}_3$, where $B'$ is a non-zero boolean ring. Nevertheless, as observed above, $(B\times \mathbb{Z}_3)G\cong BG\times \mathbb{Z}_3G$ whence both direct factors $BG$ and $\mathbb{Z}_3G$ have to be weakly nil-clean rings as being a proper homomorphic images of the former weakly nil-neat ring. However, Lemma~\ref{2.1} informs us that this cannot be happen when $G\not=\{1\}$, so this case is unrealistic, too.

As another argumentation, we may appeal to the fact that $BG$ and $\mathbb{Z}_3G$ are weakly nil-neat, which fact by referring to Lemma~\ref{2.7} will lead to a new promised contradiction, because both components $G_2,G_3$ are trivial.
\end{itemize}

\medskip

"$\Leftarrow$". If point (1) is fulfilled, we have $G=\{1\}$ and thus it is immediate that $RG\cong R$, so we are set. The validity of any of the other points (2) and (3) enables us with the aid of Lemma~\ref{2.1} that the group ring $RG$ is weakly nil-clean and so it is immediately weakly nil-neat, as required.

Now, in treating the more complicated situation (4) when $G$ is the cyclic $2$-group and $R$ is the $3$-elements field, as we already emphasized above, we will have by a straightforward direct inspection that $RG=\mathbb{Z}_3G\cong \mathbb{Z}_3\times \mathbb{Z}_3$, which is a weakly nil-neat ring by consulting with \cite{SD1}.
\end{proof}

%%%%%%%%%%%%%%%%%%%%%%%%%%%%%%%%%%%%%%%%%%%
% References
%%%%%%%%%%%%%%%%%%%%%%%%%%%%%%%%%%%%%%%%%%%
\bibliographystyle{amsplain}
%%%%%%%%%%%%%%%%%%%%%%%%%%%%%%%%%%%%%%%%%%%
% Please cite your relevant papers but at most total 5 papers/books.
%%%%%%%%%%%%%%%%%%%%%%%%%%%%%%%%%%%%%%%%%%%

\vskip2pc

\noindent{\bf Mahdi Samiei}\\
Department of Mathematics, Velayat University, Iranshahr, Iran.\\
{\tt E-mail: m.samiei@velayat.ac.ir; m.samiei91@yahoo.com}

\medskip

\noindent{\bf Peter Danchev}\\
Institute of Mathematics and Informatics, Bulgarian Academy of Sciences, Sofia, Bulgaria.\\
{\tt E-mail: danchev@math.bas.bg; pvdanchev@yahoo.com}

\end{document}